\documentclass[12pt]{article}
\usepackage{amsmath,amsfonts,amssymb,amsthm}
\usepackage{enumerate}
\usepackage{caption}
\usepackage{pbox}

\numberwithin{equation}{section}
\theoremstyle{plain}
\newtheorem{theorem}[equation]{Theorem}
\newtheorem{corollary}[equation]{Corollary}
\newtheorem{lemma}[equation]{Lemma}
\newtheorem{proposition}[equation]{Proposition}

\theoremstyle{definition}
\newtheorem{definition}[equation]{Definition}

\newtheorem{example}[equation]{Example}
\newtheorem{remark}[equation]{Remark}

\numberwithin{equation}{section}

\newcommand{\R}{{\mathbb R}}
\newcommand{\N}{{\mathbb N}}

\newcommand{\Om}{\Omega}

\providecommand{\vint}[1]{\mathchoice
          {\mathop{\vrule width 5pt height 3 pt depth -2.5pt
                  \kern -9pt \kern 1pt\intop}\nolimits_{\kern -5pt{#1}}}
          {\mathop{\vrule width 5pt height 3 pt depth -2.6pt
                  \kern -6pt \intop}\nolimits_{\kern -3pt{#1}}}
          {\mathop{\vrule width 5pt height 3 pt depth -2.6pt
                  \kern -6pt \intop}\nolimits_{\kern -3pt{#1}}}
          {\mathop{\vrule width 5pt height 3 pt depth -2.6pt
                  \kern -6pt \intop}\nolimits_{\kern -3pt{#1}}}}

\newcommand{\eps}{\varepsilon}
\newcommand{\loc}{\mathrm{loc}}

\newcommand{\BV}{\mathrm{BV}}
\newcommand{\liploc}{\mathrm{Lip}_{\mathrm{loc}}}

\newcommand{\ch}{\text{\raise 1.3pt \hbox{$\chi$}\kern-0.2pt}}

\DeclareMathOperator{\capa}{Cap}
\DeclareMathOperator{\rcapa}{cap}

\DeclareMathOperator{\diam}{diam}

\DeclareMathOperator{\Lip}{Lip}
\DeclareMathOperator*{\esssup}{ess\,sup}

\DeclareMathOperator*{\esslimsup}{ess\,lim\,sup}

\DeclareMathOperator{\supp}{spt}

\begin{document}
\title{Superminimizers and a weak Cartan property for $p=1$ in metric spaces
\footnote{{\bf 2010 Mathematics Subject Classification}: 30L99, 31E05, 26B30.
\hfill \break {\it Keywords\,}: metric measure space, bounded variation,
superminimizer, fine topology, semicontinuity,
weak Cartan property
}}
\author{Panu Lahti}
\maketitle

\begin{abstract}
We study functions of least gradient as well as related superminimizers and solutions
of obstacle problems in metric spaces that are equipped with a doubling measure
and support a Poincar\'e inequality. We show a standard weak Harnack inequality and
use it
to prove semicontinuity properties of such functions.
We also study some properties of the \emph{fine topology} in the case $p=1$.
Then we combine these theories to prove a \emph{weak Cartan property} of superminimizers
in the case $p=1$, as well as a strong version at points of nonzero capacity.
Finally we employ the weak Cartan property to show that any topology that makes
the upper representative $u^{\vee}$ of every $1$-superminimizer $u$
upper semicontinuous in open sets is stronger (in some cases, strictly) than the
$1$-fine topology.

\end{abstract}

\section{Introduction}

It is well known that solutions $u$ of the $p$-Laplace equation, for
$1<p<\infty$, can be characterized as local minimizers of the $L^p$-norm of $|\nabla u|$.
This formulation has the advantage that it can be generalized to
a metric measure space, by replacing $|\nabla u|$ with
the minimal $p$-weak upper gradient $g_u$; see Section \ref{preliminaries}
for definitions and notation.
The study of such $p$-minimizers is a starting point for nonlinear potential theory,
which is now well developed even in metric spaces
that are equipped with a doubling measure
and support a Poincar\'e inequality, see especially the monograph
\cite{BB} and e.g. \cite{BB-OD,BBS2,BBS3,S2}, and also the monographs
\cite{MZ} and \cite{HKM} for the Euclidean theory and its history in the nonweighted and weighted
setting, respectively.

In the case $p=1$, instead of the $p$-energy it is natural to minimize the total variation 
among functions of bounded variation ($\BV$ functions),
and the resulting minimizers are called functions of least gradient, see e.g. \cite{BDG,MRL,MST,SWZ}
for previous works in the Euclidean setting, and \cite{HKLS,KLLS} in the metric setting.
More precisely, a function $u\in\BV_{\loc}(\Om)$ is a function of least gradient
in an open set $\Om\subset X$ if for every
$\varphi\in \BV_c(\Om)$, we have
\[
\Vert Du\Vert(\supp \varphi)\le \Vert D(u+\varphi)\Vert(\supp \varphi)
\]
Testing only with nonnegative $\varphi$ leads to the notion of \emph{1-superminimizers},
whose study is the main objective of this paper.
In the case $p>1$, much of potential theory deals with superminimizers
and the closely related concept of
superharmonic functions, which were introduced in the metric setting in \cite{KiMa}.
A notion of $1$-superharmonic functions has been studied in the Euclidean setting in \cite{SS1},
but especially in the metric setting very little is known 
about these concepts in the case $p=1$.

For consistency, we use the term \emph{1-minimizer} instead of function of least gradient.
In \cite[Theorem 4.1]{HKLS},
$1$-minimizers were shown to be continuous outside their jump sets.
In this paper we show that this is a consequence of the fact
that for super- and subminimizers $u$, the pointwise representatives $u^{\wedge}$ and $u^{\vee}$
are lower and upper semicontinuous, respectively, at \emph{every} point.
This is Theorem \ref{thm:superminimizers are lsc}.

We also study some basic properties of solutions of \emph{obstacle problems}
in the case $p=1$;
such solutions are, in particular, $1$-superminimizers.
In the Euclidean setting, obstacle problems for the $\BV$ class
have been studied in e.g. \cite{GCP,SS2,ZiZu}, and in the metric setting in \cite{KKST-DGM}.
In this paper we first prove standard De Giorgi-type and weak Harnack inequalities
for $1$-subminimizers and certain solutions of obstacle problems,
following especially \cite{BB,HKL,KKLS}, and then use these to show
the aforementioned semicontinuity property of $1$-superminimizers as well as
a similar property for solutions of obstacle problems at points where the obstacle is continuous,
see Theorem \ref{thm:semicontinuity of obstacle problems}.

While these results are of some independent interest,
our main goal is to consider certain questions of \emph{fine potential theory} when $p=1$.
In the case $p>1$ it is known that the so-called \emph{$p$-fine topology} is the coarsest
topology that makes all $p$-superharmonic functions continuous
in open sets
(alternatively upper semicontinuous, as superharmonic functions are lower semicontinuous
already with respect to the metric topology).
In Section \ref{sec:the fine topology} we define the
notion of \emph{thinness} and the resulting
fine topology in the case $p=1$,
following \cite{L}, and generalize some properties concerning, in particular,
points of nonzero capacity from the case $p>1$ to the case $p=1$.
Then in Section \ref{sec:weak cartan property} we prove the main result of this paper,
namely the following
\emph{weak Cartan property} for $1$-superminimizers.
\begin{theorem}\label{thm:weak Cartan property}
	Let $A\subset X$ and let $x\in X\setminus A$ such that $A$
	is $1$-thin at $x$.
	Then there exist $R>0$ and $u_1,u_2\in\BV(X)$ that are $1$-superminimizers in $B(x,R)$
	such that $\max\{u_1^{\wedge},u_2^{\wedge}\}=1$ in $A\cap B(x,R)$ and
	$u_1^{\vee}(x)=0=u_2^{\vee}(x)$.
\end{theorem}
The analogous property in the case $p>1$ is well known and proved in the metric setting in
\cite{BBL-WC}.
We also prove a strong version of the property,
requiring only one superminimizer, at points of nonzero capacity; this is Proposition
\ref{prop:strong Cartan property}.
Then as in the case $p>1$ we use the weak Cartan property to show that any topology that
makes the upper approximate limit $u^{\vee}$ of every $1$-superminimizer $u$
upper semicontinuous
in every open set
$\Om$ necessarily contains the $1$-fine topology. This is
Theorem \ref{thm:superminimizer top contains fine top}.
However, we observe that unlike in the case $p>1$, the converse does not hold, that is,
the $1$-fine topology does not always make $u^{\vee}$ upper semicontinuous
for $1$-superminimizers $u$;
see Example \ref{ex:counterexample}.

Our main results seem to be new even in Euclidean spaces. A key motivation
for the work is that
a weak Cartan property will be useful in considering further questions such as \emph{p-strict subsets},
fine connectedness, and the relationship between finely open and \emph{quasiopen} sets for $p=1$,
see for example \cite{BBL-CCK,BBL-SS,Lat} for the case $p>1$.

\section{Preliminaries}\label{preliminaries}

In this section we introduce most of the notation, definitions,
and assumptions employed in the paper.

Throughout this paper, $(X,d,\mu)$ is a complete metric space that is equip\-ped
with a metric $d$ and a Borel regular outer measure $\mu$ that satisfies
a doubling property.
The doubling property means that
there is a constant $C_d\ge 1$ such that
\[
0<\mu(B(x,2r))\le C_d\mu(B(x,r))<\infty
\]
for every ball $B(x,r):=\{y\in X:\,d(y,x)<r\}$ with center $x\in X$ and radius $r>0$.
Sometimes we abbreviate $B:=B(x,r)$ and $aB:=B(x,ar)$ with $a>0$; note that in metric spaces,
a ball does not necessarily have a unique center point and radius, but we will
always consider balls for which these have been specified.
We also assume that $X$ supports a $(1,1)$-Poincar\'e inequality
that will be defined below, and that $X$ consists of at least $2$ points.
By iterating the doubling condition, we obtain
for any $x\in X$ and any $y\in B(x,R)$ with $0<r\le R<\infty$ that
\begin{equation}\label{eq:homogenous dimension}
\frac{\mu(B(y,r))}{\mu(B(x,R))}\ge \frac{1}{C_d^2}\left(\frac{r}{R}\right)^{Q},
\end{equation}
where $Q>1$ only depends on the doubling constant $C_d$.
When we want to state that a constant $C$
depends on the parameters $a,b, \ldots$, we write $C=C(a,b,\ldots)$, and
we understand all constants to be strictly positive.
When a property holds outside a set of $\mu$-measure zero, we say that it holds
almost everywhere, abbreviated a.e.

A complete metric space equipped with a doubling measure is proper,
that is, closed and bounded sets are compact.
Since $X$ is proper, for any open set $\Omega\subset X$
we define $\liploc(\Omega)$ to be the space of
functions that are Lipschitz in every open $\Omega'\Subset\Omega$.
Here $\Omega'\Subset\Omega$ means that $\overline{\Omega'}$ is a
compact subset of $\Omega$. Other local spaces of functions are defined analogously.

For any set $A\subset X$ and $0<R<\infty$, the restricted spherical Hausdorff content
of codimension one is defined to be
\[
\mathcal{H}_{R}(A):=\inf\left\{ \sum_{i=1}^{\infty}
  \frac{\mu(B(x_{i},r_{i}))}{r_{i}}:\,A\subset\bigcup_{i=1}^{\infty}B(x_{i},r_{i}),\,r_{i}\le R\right\}.
\]
The codimension one Hausdorff measure of $A\subset X$ is then defined to be
\[
\mathcal{H}(A):=\lim_{R\rightarrow 0}\mathcal{H}_{R}(A).
\]

The measure theoretic boundary $\partial^{*}E$ of a set $E\subset X$ is the set of points $x\in X$
at which both $E$ and its complement have strictly positive upper density, i.e.
\[
\limsup_{r\to 0}\frac{\mu(B(x,r)\cap E)}{\mu(B(x,r))}>0\quad
  \textrm{and}\quad\limsup_{r\to 0}\frac{\mu(B(x,r)\setminus E)}{\mu(B(x,r))}>0.
\]
The measure theoretic interior and exterior of $E$ are defined respectively by
\begin{equation}\label{eq:definition of measure theoretic interior}
I_E:=\left\{x\in X:\,\lim_{r\to 0}\frac{\mu(B(x,r)\setminus E)}{\mu(B(x,r))}=0\right\}
\end{equation}
and
\begin{equation}\label{eq:definition of measure theoretic exterior}
O_E:=\left\{x\in X:\,\lim_{r\to 0}\frac{\mu(B(x,r)\cap E)}{\mu(B(x,r))}=0\right\}.
\end{equation}
Note that the space is always partitioned into the disjoint sets
$\partial^*E$, $I_E$, and $O_E$.
By Lebesgue's differentiation theorem (see e.g. \cite[Chapter 1]{Hei}),
for a $\mu$-measurable set $E$ we have $\mu(E\Delta I_E)=0$,
where $\Delta$ is the symmetric difference.

All functions defined on $X$ or its subsets will take values
in $\overline{\R}:=[-\infty,\infty]$.
By a curve we mean a rectifiable continuous mapping from a compact interval of the real line
into $X$.
A nonnegative Borel function $g$ on $X$ is an upper gradient 
of a function $u$ on $X$ if for all nonconstant curves $\gamma$, we have
\begin{equation}\label{eq:definition of upper gradient}
|u(x)-u(y)|\le \int_\gamma g\,ds,
\end{equation}
where $x$ and $y$ are the end points of $\gamma$
and the curve integral is defined by using an arc-length parametrization,
see \cite[Section 2]{HK} where upper gradients were originally introduced.
We interpret $|u(x)-u(y)|=\infty$ whenever  
at least one of $|u(x)|$, $|u(y)|$ is infinite.

In what follows, let $1\le p<\infty$.
We say that a family of curves $\Gamma$ is of zero $p$-modulus if there is a 
nonnegative Borel function $\rho\in L^p(X)$ such that 
for all curves $\gamma\in\Gamma$, the curve integral $\int_\gamma \rho\,ds$ is infinite.
A property is said to hold for $p$-almost every curve
if it fails only for a curve family with zero $p$-modulus. 
If $g$ is a nonnegative $\mu$-measurable function on $X$
and (\ref{eq:definition of upper gradient}) holds for $p$-almost every curve,
we say that $g$ is a $p$-weak upper gradient of $u$. 
By only considering curves $\gamma$ in $\Om\subset X$,
we can talk about a function $g$ being a ($p$-weak) upper gradient of $u$ in $\Om$.

Given an open set $\Om\subset X$, we define the norm
\[
\Vert u\Vert_{N^{1,p}(\Om)}:=\Vert u\Vert_{L^p(\Om)}+\inf \Vert g\Vert_{L^p(\Om)},
\]
where the infimum is taken over all $p$-weak upper gradients $g$ of $u$ in $\Om$.
The substitute for the Sobolev space $W^{1,p}$ in the metric setting is the Newton-Sobolev space
\[
N^{1,p}(\Om):=\{u:\|u\|_{N^{1,p}(\Om)}<\infty\}.
\]
We understand every Newton-Sobolev function to be defined at every $x\in \Om$
(even though $\Vert \cdot\Vert_{N^{1,p}(\Om)}$ is, precisely speaking, then only a seminorm).
It is known that for any $u\in N_{\loc}^{1,p}(\Om)$, there exists a minimal $p$-weak
upper gradient of $u$ in $\Om$, always denoted by $g_{u}$, satisfying $g_{u}\le g$ 
a.e. in $\Om$, for any $p$-weak upper gradient $g\in L_{\loc}^{p}(\Om)$
of $u$ in $\Om$, see \cite[Theorem 2.25]{BB}.

The $p$-capacity of a set $A\subset X$ is given by
\[
\capa_p(A):=\inf \Vert u\Vert_{N^{1,p}(X)},
\]
where the infimum is taken over all functions $u\in N^{1,p}(X)$ such that $u\ge 1$ in $A$.
We know that $\capa_p$ is an outer capacity, meaning that
\[
\capa_p(A)=\inf\{\capa_p(U):\,U\supset A\textrm{ is open}\}
\]
for any $A\subset X$, see e.g. \cite[Theorem 5.31]{BB}.

If a property holds outside a set
$A\subset X$ with $\capa_p(A)=0$, we say that it holds $p$-quasieverywhere,
abbreviated $p$-q.e.
If $u\in N^{1,p}(\Om)$, then $\Vert u-v\Vert_{N^{1,p}(\Om)}=0$ if and only if $u=v$
$p$-q.e. in $\Om$, see \cite[Proposition 1.61]{BB}.
By \cite[Theorem 4.3, Theorem 5.1]{HaKi} we know that if $A\subset X$,
\begin{equation}\label{eq:null sets of Hausdorff measure and capacity}
\capa_1(A)=0\quad\textrm{if and only if}\quad\mathcal H(A)=0.
\end{equation}

The variational $p$-capacity of a set $A\subset \Om$
with respect to an open set $\Om\subset X$ is given by
\[
\rcapa_p(A,\Om):=\inf \int_X g_u^p \,d\mu,
\]
where the infimum is taken over functions $u\in N^{1,p}(X)$ such that $u\ge 1$ in $A$
(equivalently, $p$-q.e. in $A$) and $u=0$ in $X\setminus \Om$;
recall that $g_u$ is the minimal $p$-weak upper gradient of $u$.
We know that $\rcapa_p$ is also an outer capacity, in the sense that if $\Omega\subset X$ is a bounded open set and $A\Subset \Omega$, then
\[
\rcapa_p(A,\Omega)=\inf\{\rcapa_p(U,\Omega):\,U\textrm{ open},\,A\subset U\subset\Omega \},
\]
see \cite[Theorem 6.19]{BB}.
It is easy to see that in the definitions of capacities, we can assume the test functions
to satisfy $0\le u\le 1$.
For basic properties satisfied by capacities, such as monotonicity and countable subadditivity, see e.g. \cite{BB}.

Next we recall the definition and basic properties of functions
of bounded variation on metric spaces, following \cite{M}.
See also e.g. \cite{AFP, EvaG92, Fed, Giu84, Zie89} for the classical 
theory in the Euclidean setting.
Let $\Om\subset X$ be an open set.
Given a function $u\in L^1_{\loc}(\Om)$, we define the total variation of $u$ in $\Om$ by
\[
\|Du\|(\Om):=\inf\left\{\liminf_{i\to\infty}\int_{\Om} g_{u_i}\,d\mu:\, u_i\in 
\Lip_{\loc}(\Om),\, u_i\to u\textrm{ in } L^1_{\loc}(\Om)\right\},
\]
where each $g_{u_i}$ is again the minimal $1$-weak upper gradient of $u_i$ in $\Om$.
(In \cite{M}, local Lipschitz constants were used instead of upper gradients, but
the properties of the total variation can be proved similarly with either definition.)
We say that a function $u\in L^1(\Om)$ is of bounded variation, 
and denote $u\in\BV(\Om)$, if $\|Du\|(\Om)<\infty$.
For an arbitrary set $A\subset X$, we define
\[
\|Du\|(A):=\inf\{\|Du\|(U):\, A\subset U,\,U\subset X
\text{ is open}\}.
\]
If $u\in L^1_{\loc}(\Om)$ and $\Vert Du\Vert(\Omega)<\infty$, $\|Du\|(\cdot)$ is
a Radon measure on $\Omega$ by \cite[Theorem 3.4]{M}.
A $\mu$-measurable set $E\subset X$ is said to be of finite perimeter if $\|D\ch_E\|(X)<\infty$, where $\ch_E$ is the characteristic function of $E$.
The perimeter of $E$ in $\Omega$ is also denoted by
\[
P(E,\Omega):=\|D\ch_E\|(\Omega).
\]
For any $u,v\in L^1_{\loc}(\Om)$, it is straightforward to show that
\begin{equation}\label{eq:BV functions form vector space}
\Vert D(u+v)\Vert(\Om)\le \Vert Du\Vert(\Om)+\Vert Dv\Vert(\Om).
\end{equation}

We have the following coarea formula from \cite[Proposition 4.2]{M}: if $\Omega\subset X$ is an open set and $u\in L^1_{\loc}(\Omega)$, then
\begin{equation}\label{eq:coarea}
\|Du\|(\Omega)=\int_{-\infty}^{\infty}P(\{u>t\},\Omega)\,dt.
\end{equation}

We will assume throughout the paper that $X$ supports a $(1,1)$-Poincar\'e inequality,
meaning that there exist constants $C_P>0$ and $\lambda \ge 1$ such that for every
ball $B(x,r)$, every $u\in L^1_{\loc}(X)$,
and every upper gradient $g$ of $u$,
we have
\[
\vint{B(x,r)}|u-u_{B(x,r)}|\, d\mu 
\le C_P r\vint{B(x,\lambda r)}g\,d\mu,
\]
where 
\[
u_{B(x,r)}:=\vint{B(x,r)}u\,d\mu :=\frac 1{\mu(B(x,r))}\int_{B(x,r)}u\,d\mu.
\]
Applying the Poincar\'e inequality to sequences of approximating locally
Lipschitz functions in the definition of the total variation gives
the following $\BV$ version:
for every ball $B(x,r)$ and every 
$u\in L^1_{\loc}(X)$, we have
\[
\vint{B(x,r)}|u-u_{B(x,r)}|\,d\mu
\le C_P r\, \frac{\Vert Du\Vert (B(x,\lambda r))}{\mu(B(x,\lambda r))}.
\]
For a $\mu$-measurable set $E\subset X$, the above implies
(see e.g. \cite[Equation (3.1)]{KoLa})
the relative isoperimetric inequality
\begin{equation}\label{eq:relative isoperimetric inequality}
\min\{\mu(B(x,r)\cap E),\,\mu(B(x,r)\setminus E)\}\le 2 C_P rP(E,B(x,\lambda r)).
\end{equation}
Moreover, from the $(1,1)$-Poincar\'e inequality, by \cite[Theorem 4.21, Theorem 5.51]{BB}
we get the following Sobolev inequality:
if $x\in X$, $0<r<\frac{1}{4}\diam X$, and $u\in N^{1,1}(X)$ with $u=0$
in $X\setminus B(x,r)$, then
\begin{equation}\label{eq:sobolev inequality}
\left(\,\vint{B(x,r)} |u|^{Q/(Q-1)}\,d\mu\right)^{(Q-1)/Q} \le C_S r \vint{B(x,r)}  g_u\,d\mu
\end{equation}
for a constant $C_S=C_S(C_d,C_P,\lambda)\ge 1$.
Then for any $x\in X$, any $0<r<\frac{1}{4}\diam X$,
and any $u\in L^1_{\loc}(X)$ with $u=0$ in $X\setminus B(x,r)$,
by applying the above to a suitable sequence approximating $u$, we obtain
\begin{equation}\label{eq:sobolev inequality for BV}
\left(\,\vint{B(x,r)} |u|^{Q/(Q-1)}\,d\mu\right)^{(Q-1)/Q}
\le C_S r \frac{\Vert Du\Vert(X)}{\mu(B(x,r))}.
\end{equation}
For any $\mu$-measurable set $E\subset B(x,r)$, this implies by H\"older's inequality
\begin{equation}\label{eq:isop inequality with zero boundary values}
\mu(E)\le C_S r P(E,X).
\end{equation}
Moreover, if $\Om\subset X$ is an open set with $\diam \Om<\frac 14 \diam X$
(meaning $\diam \Om<\infty$ in the case $\diam X=\infty$)
and $u\in L^1_{\loc}(X)$
with $u=0$ in $X\setminus\Om$, then we can take a ball
$B(x,r)\supset \Om$ with $r= \diam \Om$,
and so by \eqref{eq:sobolev inequality for BV} and H\"older's inequality
\begin{equation}\label{eq:Poincare ineq for BV with zero boundary values}
\int_{\Om} |u|\,d\mu\le C_{S}\diam \Om \Vert Du\Vert(X).
\end{equation}
 
The lower and upper approximate limits of a function $u$ on $X$ are defined respectively by
\begin{equation}\label{eq:lower approximate limit}
u^{\wedge}(x):
=\sup\left\{t\in\R:\,\lim_{r\to 0}\frac{\mu(\{u<t\}\cap B(x,r))}{\mu(B(x,r))}=0\right\}
\end{equation}
and
\begin{equation}\label{eq:upper approximate limit}
u^{\vee}(x):
=\inf\left\{t\in\R:\,\lim_{r\to 0}\frac{\mu(\{u>t\}\cap B(x,r))}{\mu(B(x,r))}=0\right\}.
\end{equation}
Unlike Newton-Sobolev functions, we understand $\BV$ functions to be
$\mu$-equivalence classes.
To consider fine properties, we need to
consider the pointwise representatives $u^{\wedge}$ and $u^{\vee}$.

\section{Superminimizers and obstacle problems}

In this section we consider superminimizers and solutions of obstacle problems
in the case $p=1$. The symbol $\Omega$ will always denote a nonempty open subset of $X$.
We denote by $\BV_c(\Om)$ the class of functions $\varphi\in\BV(\Om)$ with compact
support in $\Om$, that is, $\supp \varphi\Subset \Om$.

\begin{definition}
We say that $u\in\BV_{\loc}(\Om)$ is a $1$-minimizer  in $\Om$ if
for all $\varphi\in \BV_c(\Om)$,
\begin{equation}\label{eq:definition of 1minimizer}
\Vert Du\Vert(\supp\varphi)\le \Vert D(u+\varphi)\Vert(\supp\varphi).
\end{equation}
We say that $u\in\BV_{\loc}(\Om)$ is a $1$-superminimizer in $\Om$
if \eqref{eq:definition of 1minimizer} holds for all nonnegative $\varphi\in \BV_c(\Om)$.
We say that $u\in\BV_{\loc}(\Om)$ is a $1$-subminimizer in $\Om$ if
\eqref{eq:definition of 1minimizer} holds for all nonpositive $\varphi\in \BV_c(\Om)$,
or equivalently if $-u$ is a $1$-superminimizer in $\Om$.
\end{definition}

Equivalently, we can replace $\supp\varphi$ by any set $A\Subset \Om$ containing $\supp\varphi$
in the above definitions.
It is easy to see that if $u$ is a $1$-superminimizer and $a\ge 0$, $b\in \R$, then
$au+b$ is a $1$-superminimizer.

Given a nonempty bounded open set $\Om\subset X$, a function $\psi\colon\Om\to\overline{\R}$,
and $f\in L^1_{\loc}(X)$
with $\Vert Df\Vert(X)<\infty$, we define the class of admissible functions
\[
\mathcal K_{\psi,f}(\Om):=\{u\in\BV_{\loc}(X):\,u\ge \psi\textrm{ in }\Om\textrm{ and }u=f\textrm{ in }X\setminus\Om\}.
\]
The (in)equalities above are understood in the a.e. sense, since $\BV$ functions are only
defined up to sets of $\mu$-measure zero.
For brevity, we sometimes write $\mathcal K_{\psi,f}$ instead of $\mathcal K_{\psi,f}(\Om)$.
By using a cutoff function,
it is easy to show that $\Vert Du\Vert(X)<\infty$
for every $u\in\mathcal K_{\psi,f}(\Om)$.

\begin{definition}
We say that $u\in\mathcal K_{\psi,f}(\Om)$ is a solution of the $\mathcal K_{\psi,f}$-obstacle problem
if $\Vert Du\Vert(X)\le \Vert Dv\Vert(X)$ for all $v\in\mathcal K_{\psi,f}(\Om)$.
\end{definition}

\begin{proposition}\label{prop:existence of solutions}
If $\diam \Om<\frac 14\diam X$ and $\mathcal K_{\psi,f}(\Om)\neq \emptyset$,
then there exists a solution
of the $\mathcal K_{\psi,f}$-obstacle problem.
\end{proposition}

\begin{proof}
Pick a sequence of functions $(u_i)\subset \mathcal K_{\psi,f}(\Om)$ with
\[
\lim_{i\to\infty}\Vert Du_i\Vert(X)
=\inf\{\Vert Dv\Vert(X):\,v\in \mathcal K_{\psi,f}(\Om)\}<\infty.
\]
By the Poincar\'e inequality \eqref{eq:Poincare ineq for BV with zero boundary values}
and the subadditivity \eqref{eq:BV functions form vector space},
we have for each $i\in\N$
\begin{align*}
\int_X|u_i-f|\,d\mu
&\le C_S\diam \Om \Vert D(u_i-f)\Vert(X)\\
&\le C_S\diam \Om (\Vert Du_i\Vert(X)+\Vert Df\Vert(X)),
\end{align*}
which is a bounded sequence. Thus $(u_i-f)$ is a bounded sequence in $\BV(X)$,
and so by \cite[Theorem 3.7]{M} there exists a subsequence (not relabeled) and a function $v\in\BV(X)$
such that $u_i-f\to v$ in $L_{\loc}^1(X)$. We can select a further subsequence
(not relabeled) such that $u_i(x)-f(x)\to v(x)$ for a.e. $x\in X$.
Hence, letting $u:=v+f$, we have
$u\ge \psi$ in $\Om$ and $u=f$ in $X\setminus \Om$.
Moreover, $u_i\to u$ in $L^1_{\loc}(X)$, and then
by lower semicontinuity of the total variation with respect
to $L^1$-convergence, we get
\[
\Vert Du\Vert(X)\le \liminf_{i\to\infty}\Vert Du_i\Vert(X),
\]
and so $u$ is a solution.
\end{proof}

Unlike in the case $p>1$, solutions are not generally unique, as can be easily seen
for example by considering translates of the Heaviside function on the real line.
The following fact, which is also in stark contrast to the case  $p>1$, is often useful.

\begin{proposition}\label{prop:set solution}
If $A\subset X$ and there exists a solution of the $\mathcal K_{\ch_A,0}$-obstacle problem,
then there exists a set $E\subset X$ such that $\ch_E$ is also a solution.
\end{proposition}

\begin{proof}
Let $u$ be a solution. By the coarea formula \eqref{eq:coarea} there exists
$t\in (0,1)$ such that $P(\{u>t\},X)\le \Vert Du\Vert(X)$.
Letting $E:=\{u>t\}$, we clearly have $\ch_E\ge \ch_A$ in $\Om$ and $\ch_E=0$
in $X\setminus\Om$. Thus $\ch_E\in \mathcal K_{\ch_A,0}$ and so it is a solution.
\end{proof}

Whenever the characteristic function of a set $E$
is a solution of an obstacle problem,
for simplicity we will call $E$ a solution as well.
Similarly, if $\psi=\ch_A$ for some $A\subset X$, we let
$\mathcal K_{A, f}:=\mathcal K_{\psi, f}$.

The following simple fact will be of much use to us.
\begin{lemma}\label{lem:solutions from capacity}
	If $x\in X$, $0<r<R<\frac 18 \diam X$, and $A\subset B(x,r)$, then there exists
	$E\subset X$ that is a solution of the $\mathcal K_{A,0}(B(x,R))$-obstacle problem
	with
	\[
	P(E,X)\le \rcapa_1(A,B(x,R)).
	\]
\end{lemma}

\begin{proof}
	Fix $\eps>0$.
	Since $A\Subset B(x,R)$, clearly $\rcapa_1(A,B(x,R))<\infty$.
	By the definition of the variational capacity, we find $u\in N^{1,1}(X)$
	with $u\ge 1$ in $A$, $u=0$ in $X\setminus B(x,R)$, and 
	\[
	\rcapa_1(A,B(x,R))+\eps\ge \int_X g_u\,d\mu\ge \Vert Du\Vert(X),
	\]
	where the last inequality follows from the fact that Lipschitz functions are dense in
	$N^{1,1}(X)$, see \cite{S} or \cite[Theorem 5.1]{BB}.
	Now $u\in \mathcal K_{A,0}(\Om)$.
	Since $\eps>0$ was arbitrary, 
	by Proposition \ref{prop:existence of solutions} and Proposition \ref{prop:set solution}
	we conclude that the
	$\mathcal K_{A,0}(B(x,R))$-obstacle problem has a solution $E\subset X$
	such that $P(E,X)\le \rcapa_1(A,B(x,R))$.
\end{proof}

The following fact follows directly from the definitions.

\begin{proposition}
If $u\in\mathcal K_{\psi,f}(\Om)$ is a solution
of the $\mathcal K_{\psi,f}$-obstacle problem, then $u$
is a $1$-superminimizer in $\Om$.
\end{proposition}

Next we prove De Giorgi-type and weak Harnack inequalities for $1$-sub\-minimizers and certain
solutions of obstacle problems. The arguments we use are mostly standard and have been employed
in the metric setting previously in \cite{BB,HKL,KKLS}, but only for (quasi)minimizers or
in the case $p>1$, so we repeat the entire proofs with small modifications and some simplifications.

\begin{proposition}
Suppose $k\in\R$ and $B(x,s_2)\Subset \Om$, and
assume either that
\begin{enumerate}[{(a)}]
\item $u$ is a $1$-subminimizer in $\Om$, or
\item $\Om$ is bounded, $u$ is a solution of the
$\mathcal K_{\psi, f}(\Om)$-obstacle problem,
and $\psi\le k$ a.e. in $B(x,s_2)$.
\end{enumerate}
Then if $0<s_1<s_2$,
\begin{equation}\label{eq:De Giorgi class}
\Vert D(u-k)_+\Vert(B(x,s_1))\le \frac{2}{s_2-s_1}\int_{B(x,s_2)}(u-k)_+\,d\mu.
\end{equation}
\end{proposition}

\begin{proof}
Take a Lipschitz function $0\le \eta\le 1$ with compact support in $B(x,s_2)$, such that
$\eta=1$ in $B(x,s_1)$ and $g_{\eta}\le 2/(s_2-s_1)$.
It is straightforward to verify that $\eta (u-k)_+\in \BV_c(\Om)$.
Now if $u$ is a $1$-subminimizer (alternative (a)), we get
\begin{equation}\label{eq:first energy comparison}
\Vert Du\Vert(B(x,s_2))\le \Vert D(u-\eta(u-k)_+)\Vert(B(x,s_2)).
\end{equation}
In alternative (b), we have at a.e. point in $\Om$ either
$u-\eta(u-k)_+=u\ge \psi$ or $u-\eta(u-k)_+=(1-\eta)u+\eta k\ge \psi$, and
moreover $u-\eta(u-k)_+\in\BV_{\loc}(X)$, so then $u-\eta(u-k)_+\in\mathcal K_{\psi,f}(\Om)$,
and thus \eqref{eq:first energy comparison} again holds.
Note that $u-\eta(u-k)_+=\min\{u,k\}+(1-\eta)(u-k)_+$.
Using the coarea formula it can be shown that
$\Vert Du\Vert=\Vert D\min\{u,k\}\Vert+\Vert D(u-k)_+\Vert$ as measures in $\Om$,
see \cite[Lemma 3.5]{HKLS}, and thus we get
\begin{align*}
&\Vert D\min\{u,k\}\Vert(B(x,s_2))+\Vert D(u-k)_+\Vert(B(x,s_2))\\
&\qquad  = \Vert Du\Vert(B(x,s_2))\\
&\qquad  \le\Vert D(\min\{u,k\}+(1-\eta)(u-k)_+)\Vert(B(x,s_2))\qquad\textrm{by }\eqref{eq:first energy comparison}\\
&\qquad  \le\Vert D\min\{u,k\}\Vert(B(x,s_2))+\Vert D((1-\eta)(u-k)_+)\Vert(B(x,s_2))
\end{align*}
by \eqref{eq:BV functions form vector space}.
Since $\Vert D\min\{u,k\}\Vert(B(x,s_2))\le \Vert Du\Vert(B(x,s_2))<\infty$, we get
\begin{equation}\label{eq:second energy comparison}
\Vert D(u-k)_+\Vert(B(x,s_2))\le \Vert D((1-\eta)(u-k)_+)\Vert(B(x,s_2)).
\end{equation}
Here we have by a Leibniz rule, see \cite[Lemma 3.2]{HKLS},
\begin{align*}
&\Vert D((1-\eta)(u-k)_+)\Vert(B(x,s_2))\\
&\qquad\le \int_{B(x,s_2)}g_{\eta}(u-k)_+\,d\mu+\int_{B(x,s_2)}(1-\eta)\,d\Vert D(u-k)_+\Vert\\
&\qquad\le \frac{2}{s_2-s_1}\int_{B(x,s_2)}(u-k)_+\,d\mu+
\Vert D(u-k)_+\Vert(B(x,s_2)\setminus B(x,s_1)).
\end{align*}
Noting that also
$\Vert D(u-k)_+\Vert(B(x,s_2)\setminus B(x,s_1))<\infty$, we combine the above
with \eqref{eq:second energy comparison} to get the result.
\end{proof}

In proving the following weak Harnack inequality, we closely follow
\cite[Proposition 8.2]{BB}, where the analogous result
is proved in the case $p>1$.
Recall the definition of the exponent $Q>1$ from \eqref{eq:homogenous dimension}.

\begin{proposition}\label{prop:weak Harnack}
Let $u\in\BV(B(x,R))$ such that \eqref{eq:De Giorgi class} holds for all
$0< s_1< s_2\le R<\frac 14 \diam X$
and all $k\ge k^*\in \R$. Let $k_0\ge k^*$ and $r\in (0,R)$. Then
\[
\esssup_{B(x,r)}u\le C_1\left(\frac{R}{R-r}\right)^{Q}\vint{B(x,R)}(u-k_0)_+\,d\mu+k_0
\]
for some constant $C_1=C_1(C_d,C_P,\lambda)$.
\end{proposition}

\begin{proof}
Choose $r\le r_1< r_2\le R$ and let $\rho:=(r_1+r_2)/2$.
Let $\eta$ be a $2/(r_2-r_1)$-Lipschitz function such that $0\le \eta\le 1$,
$\eta=1$ in $B(x,r_1)$, and $\eta=0$ outside $B(x,\rho)$.
Let $l_2>l_1\ge k^*$, and
\[
A:=\{u>l_2\}\cap B(x,\rho).
\]
Now
\[
\mu(A)\le \frac{1}{l_2-l_1}\int_{B(x,\rho)}(u-l_1)_+\,d\mu
\le \frac{1}{l_2-l_1}\int_{B(x,r_2)}(u-l_1)_+\,d\mu.
\]
Let $v:=\eta (u-l_2)_+$. 
By H\"older's inequality
\begin{equation}\label{eq:using Holder}
\begin{split}
&\int_{B(x,r_1)}(u-l_2)_+\,d\mu
\le \int_{B(x,\rho)}v\,d\mu\\
&\qquad\qquad\le \left(\int_{B(x,\rho)}v^{Q/(Q-1)}\,d\mu\right)^{(Q-1)/Q}\mu(A)^{1/Q}\\
&\qquad\qquad\le\left(\int_{B(x,\rho)}v^{Q/(Q-1)}\,d\mu\right)^{(Q-1)/Q}
\left(\frac{1}{l_2-l_1}\int_{B(x,r_2)}(u-l_1)_+\,d\mu\right)^{1/Q}.
\end{split}
\end{equation}
By the Sobolev inequality \eqref{eq:sobolev inequality for BV}
(here we need $R<\frac 14 \diam X$)
\begin{equation}\label{eq:using Sobolev}
\begin{split}
&\left(\int_{B(x,\rho)}v^{Q/(Q-1)}\,d\mu\right)^{(Q-1)/Q}\\
&\qquad\quad= \left(\int_{B(x,r_2)}v^{Q/(Q-1)}\,d\mu\right)^{(Q-1)/Q}\\
&\qquad\quad\le \frac{C_{S} r_2}{\mu(B(x,r_2))^{1/Q}}\Vert Dv\Vert(X)\\
&\qquad\quad\le \frac{C_{S} r_2}{\mu(B(x,r_2))^{1/Q}}\left(\int_X \eta\,d\Vert D(u-l_2)_+\Vert
+\int_{X}g_{\eta}(u-l_2)_+\,d\mu \right),
\end{split}
\end{equation}
where the last inequality follows from a Leibniz rule, see \cite[Lemma 3.2]{HKLS}.
Note that $g_{\eta}=0$ in $X\setminus B(x,\rho)$, see \cite[Corollary 2.21]{BB}.
By using this and the assumption
of the proposition, we can estimate
\begin{align*}
&\int_X \eta\,d\Vert D(u-l_2)_+\Vert
+\int_{X}g_{\eta}(u-l_2)_+\,d\mu\\
&\qquad\qquad\le \Vert D(u-l_2)_+\Vert(B(x,\rho))
+\int_{B(x,\rho)}g_{\eta}(u-l_2)_+\,d\mu\\
&\qquad\qquad\le
\frac{2}{r_2-\rho}\int_{B(x,r_2)}(u-l_2)_+\,d\mu
+\frac{2}{r_2-r_1}\int_{B(x,\rho)}(u-l_2)_+\,d\mu \\
&\qquad\qquad\le \frac{6}{r_2-r_1}\int_{B(x,r_2)}(u-l_2)_+\,d\mu.
\end{align*}
By combining this with \eqref{eq:using Holder} and \eqref{eq:using Sobolev}, we get
(note that $l_1<l_2$)
\begin{align*}
&\int_{B(x,r_1)}(u-l_2)_+\,d\mu
\le\frac{6 C_{S} r_2}{\mu(B(x,r_2))^{1/Q}(r_2-r_1)}
\int_{B(x,r_2)}(u-l_1)_+\,d\mu\\
&\qquad\qquad\qquad\qquad\qquad \times
\left(\frac{1}{l_2-l_1}\int_{B(x,r_2)}(u-l_1)_+\,d\mu\right)^{1/Q}.
\end{align*}
Let
\[
u(k,s):=\vint{B(x,s)}(u-k)_+\,d\mu.
\]
Since $\mu(B(x,r_2))\le C_d \mu(B(x,r_1))$, letting $C_0:=6C_S C_d$ we get
\[
u(l_2,r_1)\le \frac{C_0 r_2}{(r_2-r_1)(l_2-l_1)^{1/Q}} u(l_1,r_2)^{1+1/Q}.
\]
For $i=0,1,\ldots$, let $\rho_i:=r+2^{-i}(R-r)$ and $k_i:=k_0+d(1-2^{-i})$,
where $d>0$ is chosen below.
We show by induction that $u(k_i,\rho_i)\le 2^{-i(1+Q)} u(k_0,R)$ for $i=0,1,\ldots$.
This is clearly true for $i=0$.
Assuming the claim is true for $i$, we have
\begin{align*}
u(k_{i+1},\rho_{i+1})
&\le  \frac{C_0 R}{(\rho_i-\rho_{i+1})(k_{i+1}-k_i)^{1/Q}} u(k_i,\rho_i)^{1+1/Q}\\
&\le \frac{C_0 R}{2^{-(i+1)}(R-r)d^{1/Q}2^{-(i+1)/Q}} u(k_i,\rho_i)^{1+1/Q}\\
&\le \frac{C_0 R}{(R-r)d^{1/Q}} 2^{(i+1)(1+1/Q)}\left(2^{-i(1+Q)}u(k_0,R)\right)^{1+1/Q}\\
&=2^{-(i+1)(1+Q)}u(k_0,R)
\end{align*}
if $d=2^{(Q+1)^2}(C_0 R)^Q u(k_0,R)/(R-r)^Q$ (note that we can assume $u(k_0,R)>0$),
and so the claim is true for $i+1$.
It follows that $u(k_0+d,r)=0$, meaning that $u\le k_0+d$ a.e. in $B(x,r)$.
\end{proof}

We combine the previous two propositions to get the following theorem. Recall that $\Om$
always denotes a nonempty open set.

\begin{theorem}\label{thm:weak Harnack}
Suppose $k\in\R$ and $0<R<\tfrac 14 \diam X$ with $B(x,R)\Subset \Om$, and
assume either that
\begin{enumerate}[{(a)}]
\item $u$ is a $1$-subminimizer in $\Om$, or
\item $\Om$ is bounded, $u$ is a solution of the
$\mathcal K_{\psi, f}(\Om)$-obstacle problem,
and $\psi\le k$ a.e. in $B(x,R)$.
\end{enumerate}
Then for any $0<r<R$,
\[
\esssup_{B(x,r)}u\le C_1\left(\frac{R}{R-r}\right)^{Q}\vint{B(x,R)}(u-k)_+\,d\mu+k.
\]
\end{theorem}

Unlike $p$-harmonic functions for $p>1$, $1$-minimizers are not always continuous
with any choice of representative, as demonstrated already by the
Heaviside function on the real line. However, the following semicontinuity holds.
Recall the definitions of the pointwise representatives $u^{\wedge}$ and $u^{\vee}$
from \eqref{eq:lower approximate limit} and \eqref{eq:upper approximate limit}.

\begin{theorem}\label{thm:superminimizers are lsc}
Let $u$ be a $1$-superminimizer in $\Om$. Then $u^{\wedge}\colon\Om\to (-\infty,\infty]$
is lower semicontinuous.
\end{theorem}
\begin{proof}
Let $x\in \Om$ and $R>0$ with $B(x,2R)\Subset \Om$.
By Theorem \ref{thm:weak Harnack}(a),
$u\ge \beta$ in $B(x,R)$ for some
$\beta\in\R$.
Thus $u^{\wedge}(x)\ge\beta>-\infty$.
Take a real number $t\in [\beta,u^{\wedge}(x)]$; note that we could have $u^{\wedge}(x)=\infty$.
Clearly $t-u$ is a $1$-subminimizer in $\Om$. Fix $\eps>0$.
By applying Theorem \ref{thm:weak Harnack}(a) with $k=0$, we get for any $0<r\le R$
\begin{align*}
&\esssup_{B(x,r/2)}(t-u)
\le 2^Q C_1\vint{B(x,r)}(t-u)_+\,d\mu\\
&\qquad= \frac{2^Q C_1}{\mu(B(x,r))}\Bigg(\int_{B(x,r)\cap \{t-\eps\le u<t\}}(t-u)\,d\mu
+\int_{B(x,r)\cap \{u< t-\eps\}}(t-u)
\,d\mu\Bigg)\\
&\qquad\le 2^Q C_1\eps+2^Q C_1(t-\beta)\frac{\mu(\{u<t-\eps\}\cap B(x,r))}
{\mu(B(x,r))}\\
&\qquad\to 2^Q C_1\eps\quad\textrm{as }r\to 0
\end{align*}
by the definition of the lower approximate limit $u^{\wedge}(x)$, and the fact that
$t\le u^{\wedge}(x)$.
Hence for small enough $r>0$, $u\ge t-2^Q C_1\eps-\eps$
in $B(x,r/2)$.
Thus $u^{\wedge}\ge t-2^Q C_1\eps-\eps$ in $B(x,r/2)$.
Now if $u^{\wedge}(x)\in\R$, we can choose $t=u^{\wedge}(x)$ to establish the lower
semicontinuity, whereas if $u^{\wedge}(x)=\infty$, we can choose $t\in\R$ arbitrarily
large to achieve the same.
\end{proof}

We conclude that for a $1$-minimizer $u$, $u^{\wedge}$ is lower semicontinuous and $u^{\vee}$
is upper semicontinuous.
From this we immediately get the following corollary, which was previously proved in
\cite[Theorem 4.1]{HKLS}.
We define the jump set $S_u$ of a $\BV$ function $u$ as the set where $u^{\wedge}<u^{\vee}$.

\begin{corollary}
Let $u$ be a $1$-minimizer in $\Om$. Then $u^{\vee}$ (alternatively
$u^{\wedge}$, or the \emph{precise representative} $\widetilde{u}:=(u^{\wedge}+u^{\vee})/2$)
is continuous at every $x\in\Om\setminus S_u$.
\end{corollary}

For obstacle problems in the case $p>1$, it is well known that continuity of the obstacle
implies continuity of the solution, see \cite{Far} or \cite[Theorem 8.29]{BB}.
In the case $p=1$, the best we can hope for is
lower semicontinuity of $u^{\wedge}$ (which holds for superminimizers and thus for
solutions of obstacle problems) and the upper semicontinuity of $u^{\vee}$.
These we can indeed obtain.

\begin{theorem}\label{thm:semicontinuity of obstacle problems}
Let $u$ be a solution of the $\mathcal K_{\psi,f}(\Om)$-obstacle problem.
If $x\in\Om$ and $\psi^{\vee}(x)=\esslimsup_{y\to x}\psi(y)$ (with value in $\overline{\R}$),
then $u^{\vee}$ is ($\overline{\R}$-valued) upper semicontinuous at $x$.
If $\psi^{\vee}(x)<\infty$, then $u^{\vee}$ is real-valued upper semicontinuous at $x$.
\end{theorem}

Here $\esslimsup_{y\to x}\psi(y):=\lim_{r\to 0}\esssup_{B(x,r)}\psi$.
In particular, it is enough if $\psi$ is continuous (as an $\overline{\R}$-valued function)
at $x$.
In Example \ref{ex:counterexample} we will see that $u^{\vee}$ is not always upper semicontinuous.

\begin{proof}
Assume first that $\psi^{\vee}(x)=\infty$.
Then since $u\ge \psi$, clearly $u^{\vee}(x)\ge \psi^{\vee}(x)=\infty$,
guaranteeing $\overline{\R}$-valued upper semicontinuity at $x$.

From now on, assume $\psi^{\vee}(x)<\infty$.
By the fact that $\psi^{\vee}(x)=\esslimsup_{y\to x}\psi(y)$,
there exist $k_0\in\R$ and $R_0>0$ such that $\psi\le k_0$ a.e. in $B(x,R_0)\Subset \Om$,
and so
by  Theorem \ref{thm:weak Harnack}(b),
\[
\esssup_{B(x,R_0/2)}u\le 2^{Q} C_1 \vint{B(x,R_0)}(u-k_0)_+\,d\mu+k_0=:M<\infty.
\]
We conclude that $u^{\vee}(x)<\infty$, and since also $u^{\wedge}(x)>-\infty$ by
Theorem \ref{thm:superminimizers are lsc}, we have $u^{\vee}(x)\in \R$.

Fix $\eps>0$. Since $u\ge \psi$ in $\Om$, we have $u^{\vee}(x)\ge \psi^{\vee}(x)$.
Let $k:=u^{\vee}(x)+\eps\ge \psi^{\vee}(x)+\eps$.
If $\psi^{\vee}(x)\in\R$ (respectively, $\psi^{\vee}(x)=-\infty$),
by the fact that $\psi^{\vee}(x)=\esslimsup_{y\to x}\psi(y)$ we find
$R\in (0,R_0/2]$ such that
$\psi\le \psi^{\vee}(x)+\eps\le k$ a.e. in $B(x,R)$
(respectively, $\psi \le k$ in $B(x,R)$).
Thus we can apply Theorem \ref{thm:weak Harnack}(b) to get for any $0<r\le R$
\[
\esssup_{B(x,r/2)}u\le  2^{Q}C_1
\vint{B(x,r)}(u-u^{\vee}(x)-\eps)_+\,d\mu+u^{\vee}(x)+\eps.
\]
Here
\[
\vint{B(x,r)}(u-u^{\vee}(x)-\eps)_+\,d\mu
\le (M-u^{\vee}(x))\frac{\mu(\{u>u^{\vee}(x)+\eps\}\cap B(x,r))}{\mu(B(x,r))}
\to 0
\]
as $r\to 0$ by the definition of the upper approximate limit $u^{\vee}(x)$.
Thus for sufficiently small $r>0$,
\[
\esssup_{B(x,r/2)}u\le u^{\vee}(x)+2\eps
\]
and thus $u^{\vee}\le u^{\vee}(x)+2\eps$ in $B(x,r/2)$. We conclude that
\[
\limsup_{y\to x}u^{\vee}(y)\le u^{\vee}(x)<\infty,
\]
and since
$u^{\vee}\ge u^{\wedge}>-\infty$
in $\Om$ by Theorem \ref{thm:superminimizers are lsc}, we have
established real-valued upper semicontinuity at $x$.
\end{proof}

For \emph{general} $\BV$ functions we have the following result,
which follows from \cite[Theorem 1.1]{LaSh},
and was proved earlier in the Euclidean setting in \cite[Theorem 2.5]{CDLP}.

\begin{proposition}\label{prop:quasisemicontinuity of BV}
Let $u\in\BV(X)$ and $\eps>0$. Then there exists an open set $G\subset X$ such that
$\capa_1(G)<\eps$ and
$u^{\wedge}|_{X\setminus G}$ is lower semicontinuous.
\end{proposition}

This \emph{quasi-semicontinuity} is to be compared with the quasicontinuity of
Newton-Sobolev functions: if $u\in N^{1,p}(X)$ for $1\le p<\infty$ and $\eps>0$,
then there exists an open set $G\subset X$ such that
$\capa_p(G)<\eps$ and
$u|_{X\setminus G}$ is continuous;
see \cite[Theorem 1.1]{BBS} or \cite[Theorem 5.29]{BB}.

\section{The $1$-fine topology}\label{sec:the fine topology}

In this section we consider some basic properties of the \emph{$1$-fine topology}.
The following definition is from \cite{L}.

\begin{definition}\label{def:1 fine topology}
	We say that $A\subset X$ is $1$-thin at the point $x\in X$ if
	\[
	\lim_{r\to 0}r\frac{\rcapa_1(A\cap B(x,r),B(x,2r))}{\mu(B(x,r))}=0.
	\]
	If $A$ is not $1$-thin at $x$, we say that it is $1$-thick.
	We also say that a set $U\subset X$ is $1$-finely open if $X\setminus U$ is $1$-thin at every $x\in U$. Then we define the $1$-fine topology as the collection of $1$-finely open subsets
	of $X$.
\end{definition}

See \cite[Lemma 4.2]{L}
for a proof of the fact that the $1$-fine topology is indeed a topology.

We record the following fact given in \cite[Lemma 11.22]{BB}, and use it to prove two
lemmas that will be useful later.
\begin{lemma}\label{lem:capacity wrt different balls}
Let $x\in X$, $r>0$, and $A\subset B(x,r)$. Then for every $1<s<t$
with $tr<\frac 14 \diam X$, we have
\[
\rcapa_1(A,B(x,tr))\le \rcapa_1(A,B(x,sr))\le C_S\left(1+\frac{t}{s-1}\right)\rcapa_1(A,B(x,tr)),
\]
where $C_S$ is the constant in the Sobolev inequality \eqref{eq:sobolev inequality}.
\end{lemma}

\begin{lemma}\label{lem:discrete thinness condition}
Let $A\subset X$, $x\in X$, $R>0$, and $M>1$ such that
\[
\lim_{i\to \infty}M^{-i}R\frac{\rcapa_1(A\cap B(x,M^{-i}R),B(x,2M^{-i}R))}
{\mu(B(x,M^{-i}R))}=0.
\]
Then
\[
\lim_{r\to 0}r\frac{\rcapa_1(A\cap B(x,r),B(x,2r))}{\mu(B(x,r))}=0.
\]
\end{lemma}
\begin{proof}
If $i\in\N$ such that $2M^{-i}R<\tfrac 14 \diam X$ and $r\in [M^{-i-1}R,M^{-i}R]$,
we have by Lemma
\ref{lem:capacity wrt different balls}
\begin{align*}
&r\frac{\rcapa_1(A\cap B(x,r),B(x,2r))}{\mu(B(x,r))}\\
&\qquad\le C_S(1+2M) r\frac{\rcapa_1(A\cap B(x,r),B(x,2M^{-i}R))}{\mu(B(x,r))}\\
&\qquad\le C_S(1+2M) C_d^{\lceil\log_2 M \rceil} M^{-i}R\frac{\rcapa_1(A\cap B(x,M^{-i}R),B(x,2M^{-i}R))}{\mu(B(x,M^{-i}R))},
\end{align*}
where $\lceil a\rceil$ denotes the smallest integer at least $a$. From this the claim follows.
\end{proof}

The following is a standard result in the case $p>1$,
see e.g. \cite[Lemma 12.11]{HKM} or \cite[Lemma 4.7]{BBL-WC},
and we prove it similarly for $p=1$.
\begin{lemma}\label{lem:open modification of thin sets}
	Let $A\subset X$ and $x\in X\setminus A$.
	If $A$ is $1$-thin at $x$, there exists an open set $W\supset A$
	that is $1$-thin at $x$.
\end{lemma}

\begin{proof}
	Let $B_i:=B(x,2^{-i})$, $i\in\N$.
	By Lemma \ref{lem:capacity wrt different balls}, if $2^{-i+2}<\tfrac 14 \diam X$, then
	\[
	\rcapa_1(A\cap\overline{B_i},2B_i)\le 5C_S\rcapa_1(A\cap\overline{B_i},4B_i)
	\le 5C_S\rcapa_1(A\cap 2B_i,4B_i).
	\]
	By the fact that $\rcapa_1$ is an outer capacity, for each $i\in\N$ we find an open set
	$W_i\supset A\cap \overline{B_i}$ such that
	\[
	2^{-i}\frac{\rcapa_1(W_i,2B_i)}{\mu(B_i)}
	\le 2^{-i}\frac{\rcapa_1(A\cap\overline{B_i},2B_i)}{\mu(B_i)}+1/i.
	\]
	Let
	\[
	W:=(X\setminus \overline{B_1})\cup (W_1\setminus \overline{B_2}) 
	\cup (W_1\cap W_2\setminus \overline{B_3})
	\cup (W_1\cap W_2\cap W_3\setminus \overline{B_4})
	\cup\ldots
	\]
	Now $W$ is open and $A\subset W$, and $W\cap B_i\subset W_i$ for all $i\in\N$.
	Thus by combining the two inequalities above, we get
	for any $i\in\N$ with $2^{-i+2}<\tfrac 14 \diam X$,
	\begin{align*}
	2^{-i}\frac{\rcapa_1(W\cap B_i,2B_i)}{\mu(B_i)}
	& \le 2^{-i}\frac{\rcapa_1(W_i,2B_i)}{\mu(B_i)}\\
	&\le 2^{-i}\frac{\rcapa_1(A\cap\overline{B_i},2B_i)}{\mu(B_i)}+1/i\\
	&\le 5C_S C_d 2^{-i+1}\frac{\rcapa_1(A\cap 2B_i,4B_i)}{\mu(2B_i)}+1/i\\
	&\to 0\quad\textrm{as }i\to\infty
	\end{align*}
	by the fact that $A$ is $1$-thin at $x$.
	By Lemma \ref{lem:discrete thinness condition} we conclude that $W$ is also $1$-thin at $x$.
\end{proof}

The analog of the next proposition is again known for $p>1$,
see \cite[Proposition 1.3]{BBL-WC}, but in this case our proof will be rather different.
In the case $p>1$ the proof relies on the theory of $p$-harmonic functions,
but we are able to use a more direct argument that relies on the relative isoperimetric inequality.

\begin{proposition}\label{prop:positive capacity implies thickness}
Let $x\in X$ with $\capa_1(\{x\})>0$. Then $\{x\}$ is $1$-thick at $x$.
\end{proposition}

Towards proving the proposition, we first collect some more facts.
According to \cite[Proposition 6.16]{BB}, if $x\in X$, $0<r<\frac 18\diam X$,
and $A\subset B(x,r)$, then
for some constant $C=C(C_d,C_P,\lambda)$,
\begin{equation}\label{eq:comparison of capacities}
\frac{\capa_1(A)}{C(1+r)}\le\rcapa_1(A,B(x,2r))\le 2\left(1+\frac 1r\right)\capa_1(A).
\end{equation}
In fact, the proof reveals that the second
inequality holds with any $r>0$.
We will need one more estimate for the variational $1$-capacity;
recall the definition of the measure theoretic interior $I_A$ from
\eqref{eq:definition of measure theoretic interior}.

\begin{lemma}\label{lem:capacity estimate}
Let $x\in X$, $0<r<\frac 18 \diam  X$, and $A\subset B(x,r)$ with $x\in I_A$.
Then there exists $s_r\in (0,r]$ such that
\[
\frac{\mu(B(x,s_r))}{C_2 s_r}\le \rcapa_1(A,B(x,2r))
\]
for a constant $C_2=C_2(C_d,C_P)$.
\end{lemma}

\begin{proof}
By Lemma \ref{lem:solutions from capacity} we find a set $E\subset B(x,2r)$ such that
$A\subset E$ and
\begin{equation}\label{eq:choice of E when x is in IA}
P(E,X)\le \rcapa_1(A,B(x,2r)).
\end{equation}
By the doubling property of the measure and the fact that
$0<r<\tfrac 18 \diam X$,
there exists $\beta=\beta(C_d)\in (1/2,1)$
such that
\[
\mu(E)\le \mu(B(x,2r))\le \beta\mu(B(x,4r)),
\]
see \cite[Lemma 3.7]{BB}.
Now pick the first number $i=0,1,\ldots$ such that
\[
\mu(E\cap B(x,2^{-i+1}r))\ge \frac 12 \mu(B(x,2^{-i+1}r));
\]
such $i$ exists by the fact that $x\in I_A\subset I_E$.
If $i=0$, then
\[
\frac {1}{2C_d} \mu(B(x,4r))\le \frac {1}{2} \mu(B(x,2r))\le\mu(E) \le\beta\mu(B(x,4r)).
\]
If $i\ge 1$,
then
\[
\mu(E\cap B(x,2^{-i+2}r))<\frac 12 \mu(B(x,2^{-i+2}r)),
\]
but also
\begin{align*}
\mu(E\cap B(x,2^{-i+2}r))
\ge  \mu(E\cap B(x,2^{-i+1}r))
&\ge \frac {1}{2} \mu(B(x,2^{-i+1}r))\\
&\ge \frac {1}{2C_d} \mu(B(x,2^{-i+2}r)).
\end{align*}
Letting $s:=2^{-i+2}r$, in both cases
\[
\frac {1}{2C_d} \mu(B(x,s))\le \mu(E\cap B(x,s))\le \beta \mu(B(x,s)).
\]
By the relative isoperimetric inequality \eqref{eq:relative isoperimetric inequality},
\[
2C_P s P(E,(B(x,\lambda s)))\ge \min\left\{\beta,1-\frac{1}{2C_d}\right\}\mu(B(x,s)).
\]
Thus by \eqref{eq:choice of E when x is in IA},
\begin{align*}
\rcapa_1(A,B(x,2r))
&\ge P(E,X)\\
&\ge P(E,B(x,\lambda s))\\
&\ge (2C_P)^{-1}\min\left\{\beta,1-\frac{1}{2C_d}\right\}
\frac{\mu(B(x,s))}{s}\\
&\ge (8C_P)^{-1}\min\left\{\beta,1-\frac{1}{2C_d}\right\}
\frac{\mu(B(x,s/4))}{s/4}.
\end{align*}
Thus we can choose $s_r:=s/4$.
\end{proof}

We also need the following simple lemma.

\begin{lemma}\label{lem:real line lemma}
Suppose $f\colon (0,\infty)\to (0,\infty)$ such that
$\lim_{r\to 0^+}f(r)=0$. Pick $s_r\in (0,r]$ for every $r>0$. Then
\[
\limsup_{r\to 0}\frac{f(r)}{f(s_r)}\ge 1.
\]
\end{lemma}
\begin{proof}
Fix $\eps>0$.
For every sufficiently small $R>0$ we find $0<r< R$ such that $f(r)>\sup_{0<s<R}f(s)/(1+\eps)$.
Then also $f(r)>f(s_r)/(1+\eps)$, and letting $\eps\to 0$ we get the result.
\end{proof}

\begin{proof}[Proof of Proposition \ref{prop:positive capacity implies thickness}]
First assume that
\[
\limsup_{r\to 0}\frac{r}{\mu(B(x,r))}>0.
\]
By \eqref{eq:comparison of capacities} we have
\begin{align*}
\limsup_{r\to 0} r\frac{\rcapa_1(\{x\}\cap B(x,r),B(x,2r))}{\mu(B(x,r))}
\ge \limsup_{r\to 0}r\frac{\capa_1(\{x\})}{C(1+r)\mu(B(x,r))}>0,
\end{align*}
so $\{x\}$ is $1$-thick at $x$.

Then suppose
\begin{equation}\label{eq:r divided by measure of ball goes to zero}
\lim_{r\to 0}\frac{r}{\mu(B(x,r))}=0.
\end{equation}
(Note that this is possible by the Example below.)
Let $0<r<\frac 18 \diam X$. By the fact that $\rcapa_1$ is an outer capacity,
we find $0<t\le r$ such that
\[
\rcapa_1(\{x\},B(x,2r))\ge
\rcapa_1(B(x,t),B(x,2r))-1.
\]
By Lemma \ref{lem:capacity estimate} we find $s_r\in (0,r]$ such that
\[
\rcapa_1(B(x,t),B(x,2r))\ge \frac{\mu(B(x,s_r))}{C_2 s_r}.
\]
Combining these,
\[
r\frac{\rcapa_1(\{x\},B(x,2r))}{\mu(B(x,r))}
\ge \frac{r}{\mu(B(x,r))}\frac{\mu(B(x,s_r))}{C_2 s_r}-\frac{r}{\mu(B(x,r))}.
\]
Letting $f(r):=r/\mu(B(x,r))$, we get by \eqref{eq:r divided by measure of ball goes to zero}
and Lemma \ref{lem:real line lemma}
\[
\limsup_{r\to 0}r\frac{\rcapa_1(\{x\},B(x,2r))}{\mu(B(x,r))}
= \limsup_{r\to 0}\frac{f(r)}{C_2 f(s_r)}\ge \frac{1}{C_2}>0,
\]
so that $\{x\}$ is $1$-thick at $x$.
\end{proof}

\begin{example}\label{ex:singular points}
	Let $X=\R^n$ equipped with the Euclidean metric and the weighted Lebesgue measure
	$d\mu:=w\,d\mathcal L^n$, with $w=|x|^{a}$ for $a\in(-n,-n+1)$.
	It is straightforward to check that $w$ is a Muckenhoupt $A_1$-weight,
	and thus $\mu$ is doubling and supports a $(1,1)$-Poincar\'e inequality,
	see e.g. \cite[Chapter 15]{HKM} for these concepts.  
	Denoting the origin by $0$, we have
	\[
	\lim_{r\to 0}\frac{r}{\mu(B(0,r))}=0,
	\]
	demonstrating that this possibility needs to be taken into account.
\end{example}

Now we derive a converse type of result compared with
Lemma \ref{lem:solutions from capacity}, given in Lemma
\ref{lem:capacity estimated by perimeter} below.

\begin{lemma}[{\cite[Lemma 4.3]{L}}]\label{lem:local boxing inequality}
	Let $x\in X$, $r>0$, and let $E\subset X$ be a $\mu$-measurable set with
	\begin{equation}\label{eq:little E}
	\frac{\mu(E\cap B(x,2r))}{\mu(B(x,2r))}\le \frac{1}{2C_d^{\lceil\log_2 ( 128\lambda)\rceil}}.
	\end{equation}
	Then for some constant $C_3=C_3(C_d,C_P,\lambda)$, 
	\begin{equation}\label{eq:capacity of measure theoretic interior}
	\rcapa_1 (I_E\cap B(x,r),B(x,2r))\le C_3 P(E,B(x,2 r)).
	\end{equation}
\end{lemma}

We can strengthen this in the following way.

\begin{lemma}\label{lem:capacity estimated by perimeter}
	Let $x\in X$, $r>0$, and let $E\subset X$ be a $\mu$-measurable set with
	\[
	\frac{\mu(E\cap B(x,2r))}{\mu(B(x,2r))}\le \frac{1}{2C_d^{\lceil\log_2 ( 128\lambda)\rceil}}.
	\]
	Then
	\[
	\rcapa_1((I_E\cup\partial^*E)\cap B(x,r),B(x,2r))\le C_3 P(E,B(x,2r)),
	\]
	where $C_3$ is the constant from Lemma \ref{lem:local boxing inequality}.
\end{lemma}
\begin{proof}
	By Lemma \ref{lem:local boxing inequality},
	\eqref{eq:capacity of measure theoretic interior} holds.
	We can assume that $P(E,B(x,2r))<\infty$.
	Fix $\eps>0$. By the definition of the variational capacity, we find a function
	$v\in N^{1,1}(X)$ with $v= 1$ in $I_E\cap B(x,r)$,
	$v=0$ in $X\setminus B(x,2r)$, and
	\begin{equation}\label{eq:estimate for upper gradient of v}
	\int_X g_v\,d\mu\le \rcapa_1(I_E\cap B(x,r),B(x,2r))+\eps.
	\end{equation}
	Since $1$-q.e. point is a Lebesgue point of $v$, see \cite[Theorem 4.1, Remark 4.2]{KKST2},
	we have
	$v(x)= 1$ for $1$-q.e. $x\in \partial^*E\cap B(x,r)$.
	Thus by \eqref{eq:estimate for upper gradient of v} and
	\eqref{eq:capacity of measure theoretic interior},
	\begin{align*}
	\rcapa_1((I_E\cup \partial^*E)\cap B(x,r),B(x,2r))
	&\le \int_X g_v\,d\mu\\
	&\le \rcapa_1(I_E\cap B(x,r),B(x,2r))+\eps\\
	&\le C_3 P(E,B(x,2r))+\eps.
	\end{align*}
	Letting $\eps\to 0$, we get the result.
	\end{proof}

Finally, we record the following consequence of \cite[Theorem 5.2]{L}.

\begin{theorem}\label{thm:fine semicontinuity of BV functions}
Let $u\in\BV(X)$. Then $u^{\wedge}$ is $1$-finely lower semicontinuous at $1$-q.e.
$x\in X$.
\end{theorem}

In other words, for $1$-q.e. $x\in X$, every set $\{u^{\wedge}>t\}$
(with $t\in\R$) that contains $x$ is
a $1$-fine neighborhood of $x$.
For Newton-Sobolev functions we have the stronger result that if $u\in N^{1,p}(X)$
for $1< p<\infty$, then $u$ is $p$-finely continuous at $p$-q.e. $x\in X$,
see \cite{JB}, \cite{Kor}, or \cite[Theorem 11.40]{BB}; we do not give the definition of the $p$-fine
topology for $p>1$ here but it can also be found in the above references.

\section{The weak Cartan property}\label{sec:weak cartan property}

In this section we prove the weak Cartan property, as well as a strong version at points
of nonzero  $1$-capacity.
Our proof will rely on breaking the set $A$ into two subsets
that do not intersect certain annuli around $x$. Such a separation argument
is inspired by the proof of the analogous property in the case $p>1$,
see \cite{BBL-WC}, which in turn is based on \cite{HKM-FT} and \cite{LM}.

\begin{lemma}\label{lem:smallness in annuli}
Let $B=B(x,R)$ be a ball with $0<R<\frac{1}{12} \diam X$, and
suppose that $A\subset B$ with $A\cap (\tfrac{9}{20} B\setminus \tfrac{1}{4} B)=\emptyset$.
Let $E\subset X$ be a solution of the $\mathcal K_{A,0}(\tfrac 32 B)$-obstacle problem
(as guaranteed by Lemma \ref{lem:solutions from capacity}).
Then for all
$y\in \tfrac{2}{5} B \setminus  \tfrac{5}{16} B$,
\[
\ch_E^{\vee}(y)\le C_4 R \frac{\rcapa_1(A,2B)}{\mu(B)}
\]
for some constant $C_4=C_4(C_d,C_P,\lambda)$.
\end{lemma}

\begin{proof}
By Lemma \ref{lem:solutions from capacity} and Lemma \ref{lem:capacity wrt different balls}
we know that
\[
P(E,X)\le \rcapa_1(A,\tfrac 32 B) \le 5C_S \rcapa_1(A,2B),
\]
and thus by the isoperimetric inequality \eqref{eq:isop inequality with zero boundary values},
\begin{equation}\label{eq:E1 has small measure}
\mu(E)\le 2C_S R P(E,X)\le 10C_S^2 R \rcapa_1(A,2B).
\end{equation}
For any $z\in \tfrac{2}{5} B\setminus \tfrac{5}{16} B$, letting $r:=R/20$ we have
$B(z,r)\subset \tfrac{9}{20} B\setminus \tfrac{1}{4}B$, and so by
Theorem \ref{thm:weak Harnack}(b),
\begin{align*}
\sup_{B(z,r/2)} \ch_E^{\vee }
&\le \esssup_{B(z,r/2)}\ch_E\\
&\le C_1\left(\frac{r}{r-r/2}\right)^Q\vint{B(z,r)}(\ch_E)_+\,d\mu\\
&= \frac{2^Q C_1}{\mu(B(z,r))}\int_{B(z,r)} (\ch_E)_+\,d\mu\\
&\le \frac{2^Q C_1 C_d^6}{\mu(B)}\mu(E)\\
&\le 5\times 2^{Q+1} C_1 C_d^6  C_S^2 R \frac{\rcapa_1(A,2B)}{\mu(B)}
\end{align*}
by \eqref{eq:E1 has small measure}. Thus we can choose $C_4=5\times 2^{Q+1} C_1 C_d^6 C_S^2$.
\end{proof}

Now we prove the weak Cartan property, Theorem \ref{thm:weak Cartan property}.
In fact, we give the following formulation containing somewhat
more information, which will be useful in future work when considering
\emph{p-strict subsets} and a \emph{Choquet property} in the case $p=1$,
cf. \cite[Lemma 3.3]{BBL-SS}, \cite[Lemma 2.6]{Lat}, and \cite{BBL-CCK}.

\begin{theorem}\label{thm:weak Cartan property in text}
	Let $A\subset X$ and let $x\in X\setminus A$ be such that $A$
	is $1$-thin at $x$.
	Then there exist $R>0$ and $E_0,E_1\subset X$ such that $\ch_{E_0},\ch_{E_1}\in\BV(X)$,
	$\ch_{E_0}$ and $\ch_{E_1}$ are $1$-superminimizers in $B(x,R)$,
	$\max\{\ch_{E_0}^{\wedge},\ch_{E_1}^{\wedge}\}=1$ in $A\cap B(x,R)$,
	$\ch_{E_0}^{\vee}(x)=0=\ch_{E_1}^{\vee}(x)$,
	$\{\max\{\ch_{E_0}^{\vee},\ch_{E_1}^{\vee}\}>0\}$ is $1$-thin at $x$,
	and
	\begin{equation}\label{eq:weak Cartan property energy thinness}
	\lim_{r\to 0}r\frac{P(E_0,B(x,r))}{\mu(B(x,r))}=0,\qquad
	\lim_{r\to 0}r\frac{P(E_1,B(x,r))}{\mu(B(x,r))}=0.
	\end{equation}
\end{theorem}

\begin{proof}
By Lemma \ref{lem:open modification of thin sets}
we find an open set $W\supset A$ that is $1$-thin at $x$.
Fix $0<R<\frac{1}{12} \diam X$ such that
\[
\sup_{0<s\le R}s\frac{\rcapa_1(W\cap B(x,s),B(x,2s))}{\mu(B(x,s))}<\frac{1}{2C_4.}
\]
Let $B_i:=B(x,2^{-i}R)$ and let $H_i:=B_i\setminus \frac{9}{10}\overline{B_{i+1}}$,
$i=0,1,\ldots$.
Then let
\[
D_i:=\bigcup_{j=i,\, i+2,\, i+4,\ldots} H_j,\qquad i=0,1,\ldots,
\]	
so that $D_0\cup D_1=B(x,R)$.
Let $W_i:=W\cap D_i$, $i=0,1,\ldots$, and then by Lemma \ref{lem:solutions from capacity}
we can let $E_i\subset X$
be a solution of the $\mathcal K_{W_i,0}(\frac 32 B_{i})$-obstacle problem;
clearly $\ch_{E_i}\in\BV(X)$ for all $i$.
Let $F_i:=\frac 45 B_i\setminus \frac 54 B_{i+1}\subset H_i$, $i\in\N$.
Fix $i=0,1,\ldots$. From Lemma \ref{lem:smallness in annuli}
we get for all $y\in F_{i+1}$
\[
\ch_{E_i}^{\vee}(y)\le C_4 2^{-i}R \frac{\rcapa_1(W\cap B_i,2B_{i})}{\mu(B_i)}\le
\frac 12.
\]
Since $\ch_{E_i}^{\vee}$ can only take the values $0,1$, we conclude that
$\ch_{E_i}^{\vee}=0$ in $F_{i+1}$, and thus by the Lebesgue differentiation theorem,
\begin{equation}\label{eq:disconnectedness of Ei}
\mu(E_i\cap F_{i+1})=0.
\end{equation}
Note that
$E_{i+2}\cup (E_i\setminus \tfrac 45 B_{i+1})$
is admissible for the $\mathcal K_{W_i,0}(\frac 32 B_{i})$-obstacle problem.
Now if we had
\[
P(E_{i+2},X)<P(E_i\cap \tfrac 54 B_{i+2},X),
\]
then by the fact that the sets $E_{i+2}\subset \frac 32 B_{i+2}$ and
$E_i\setminus \frac 45 B_{i+1}$
are separated by a strictly positive distance,
\begin{align*}
P(E_{i+2}\cup (E_i\setminus \tfrac 45 B_{i+1}),X)
&=P(E_{i+2},X)+P(E_i\setminus \tfrac 45 B_{i+1},X)\\
& <P(E_i\cap \tfrac 54 B_{i+2},X)+P(E_i\setminus \tfrac 45 B_{i+1},X)\\
& =P(E_i,X)
\end{align*}
by \eqref{eq:disconnectedness of Ei},
which would contradict the fact that $E_i$ is a solution of the
$\mathcal K_{W_{i},0}(\tfrac 32 B_{i})$-obstacle problem.
Thus $P(E_{i+2},X)\ge P(E_i\cap \tfrac 54 B_{i+2},X)$,
and since $E_i\cap \frac 54 B_{i+2}$
is admissible for the $\mathcal K_{W_{i+2},0}(\frac 32 B_{i+2})$-obstacle problem,
we conclude that it is a solution.
Inductively, we find that $E_0\cap \frac 54 B_{i}$
is a solution of the
$\mathcal K_{W_{i},0}(\tfrac 32 B_{i})$-obstacle problem, for any $i=2,4,6,\ldots$.
Analogously, $E_1\cap \frac 54 B_{i}$
is a solution of the
$\mathcal K_{W_{i},0}(\tfrac 32 B_{i})$-obstacle problem,
for any $i=3,5,7,\ldots$.

By Lemma \ref{lem:solutions from capacity} and the fact that
$E_0\cap \frac 54 B_{i}$
is a solution of the $\mathcal K_{W_{i},0}(\tfrac 32 B_{i})$-obstacle problem,
and by Lemma \ref{lem:capacity wrt different balls}, we have
\begin{equation}\label{eq:initial perimeter estimate E0}
P(E_0\cap \tfrac 54 B_{i},X)
\le \rcapa_1(W_i,\tfrac 32 B_i)
\le 5C_S \rcapa_1(W\cap B_i,2B_i)
\end{equation}
for every $i=2,4,6,\ldots$, and similarly
$P(E_1\cap \tfrac 54 B_{i},X)\le 5C_S \rcapa_1(W\cap B_i,2B_i)$
for every $i=3,5,7,\ldots$.

Let $0<\delta<(20 C_S^2 C_d^{\lceil\log_2(128\lambda)\rceil})^{-1}$.
Since $W$ is $1$-thin at $x$, for some even $m\in\N$ and every $i=m,m+2,\ldots$,
we have
\[
2^{-i}R\frac{\rcapa_1(W\cap B_i,2B_i)}{\mu(B_i)}
\le\delta.
\]
Fix such $m$. Together with \eqref{eq:initial perimeter estimate E0},
this gives
\begin{equation}\label{eq:perimeter estimate E0}
2^{-i}R\frac{P(E_0\cap \tfrac 54 B_{i},X)}{\mu(B_i)}
\le 5C_S \delta
\end{equation}
for every $i=m,m+2,\ldots$.
	By the isoperimetric inequality
	\eqref{eq:isop inequality with zero boundary values},
	we now have
	\[
	\mu(E_0\cap \tfrac 54 B_{i})
	\le C_S 2^{-i+1}R P(E_0\cap \tfrac 54 B_{i},X)
	\le  10 C_S^2 \delta \mu(B_i).
	\]
	Thus
	\begin{equation}\label{eq:estimate for measure of E0}
	\frac{\mu(E_0\cap \tfrac 54 B_{i})}{\mu(2B_i)}
	\le \frac{\mu(E_0\cap \tfrac 54 B_{i})}{\mu(B_i)}
	\le 10 C_S^2 \delta
	\le \frac{1}{2 C_d^{\lceil\log_2(128\lambda)\rceil}}.
	\end{equation}
By the fact that
\[
\{\ch_{E_0}^{\vee}=1\}\cap B_i=(I_{E_0}\cup \partial^*{E_0})\cap B_i=
\big(I_{E_0\cap \frac 54 B_{i}}\cup \partial^*(E_0\cap \tfrac 54 B_{i})\big)\cap B_i
\]
and Lemma \ref{lem:capacity estimated by perimeter}, we get
\begin{align*}
&2^{-i}R\frac{\rcapa_1(\{\ch_{E_0}^{\vee}=1\}\cap B_i,2B_i)}{\mu(B_{i})}\\
&\qquad\qquad  =2^{-i}R\frac{\rcapa_1\left(\big(I_{E_0\cap \frac 54 B_{i}}\cup 
\partial^*(E_0\cap \tfrac 54 B_{i})\big)\cap B_i,2 B_{i}\right)}{\mu(B_{i})}
\\
&\qquad\qquad  \le 2^{-i}R C_3 \frac{P(E_0\cap \tfrac 54 B_{i},X)}{\mu(B_{i})}\\
&\qquad\qquad  \le 5C_3 C_S \delta
\end{align*}
by \eqref{eq:perimeter estimate E0}.
Since this holds for every $i=m,m+2,\ldots$, and since $\delta$ can be made
arbitrarily small, by Lemma \ref{lem:discrete thinness condition} we obtain
\[
\lim_{r\to 0}r\frac{\rcapa_1(\{\ch_{E_0}^{\vee}=1\}\cap B(x,r),B(x,2r))}{\mu(B(x,r))}
=0.
\]
Analogously, we prove the corresponding result for $E_1$.
Since $\ch_{E_0}^{\vee}>0$ exactly when $\ch_{E_0}^{\vee}=1$,
we have established that
$\{\max\{\ch_{E_0}^{\vee},\ch_{E_1}^{\vee}\}>0\}$ is $1$-thin at $x$.
Since $\delta$ can be chosen arbitrarily small also in \eqref{eq:estimate for measure of E0},
we get $\ch_{E_0}^{\vee}(x)=0$, and similarly $\ch_{E_1}^{\vee}(x)=0$.
Moreover, since $A\cap D_0\subset W_0\subset E_0$ and $W_0$ is open,
$\ch_{E_0}^{\wedge}=1$ in $A\cap D_0$.
Analogously, $\ch_{E_1}^{\wedge}=1$ in $A\cap D_1$,
so that $\max\{\ch_{E_0}^{\wedge},\ch_{E_1}^{\wedge}\}= 1$ in $A\cap B(x,R)$.
Finally, \eqref{eq:weak Cartan property energy thinness} follows easily from
\eqref{eq:perimeter estimate E0} (and the corresponding property for $E_1$).
\end{proof}

It can be noted that in the case $p>1$, the proof of the weak Cartan property relies
on the comparison principle as well as weak
Harnack inequalities for both superminimizers and subminimizers.
We only have the last of these three tools available, but we are able to replace the others
(and in fact get a simpler argument)
with the very powerful fact that the superminimizer functions can be taken to be
characteristic functions of sets of finite perimeter;
recall especially \eqref{eq:disconnectedness of Ei}.

\begin{proof}[Proof of Theorem \ref{thm:weak Cartan property}]
	Let $R>0$ and $E_0,E_1\subset X$ as given by Theorem \ref{thm:weak Cartan property in text},
	and choose $u_1:=\ch_{E_0}$ and $u_2:=\ch_{E_1}$.
\end{proof}

The analog of the following result is again known in the case $p>1$, see \cite[Lemma 6.2]{BB}.
Our proof will be similar, but we need to rely on the quasisemicontinuity of $\BV$ functions
instead of the quasicontinuity that is available in the case $p>1$.

\begin{proposition}\label{prop:cap goes to zero at thinness points}
	Let $A\subset X$ be $1$-thin at $x\in X$ and let $R_0>0$.
	Then
	\[
	\lim_{r\to 0}\rcapa_1(A\cap B(x,r),B(x,R_0))=0.
	\]
\end{proposition}

Note that this does not follow directly from the definition of $1$-thinness,
since it is possible that $r/\mu(B(x,r))\to 0$ as $r\to 0$, recall Example \ref{ex:singular points}.

\begin{proof}
	First assume that $\capa_1(\{x\})=0$. Then $\rcapa_1(\{x\},B(x,R_0))=0$
	by \eqref{eq:comparison of capacities}, and so by the fact that $\rcapa_1$ is an outer
	capacity,
	\begin{align*}
	\limsup_{r\to 0}\rcapa_1(A\cap B(x,r),B(x,R_0))
	&\le \limsup_{r\to 0}\rcapa_1(B(x,r),B(x,R_0))\\
	&=\rcapa_1(\{x\},B(x,R_0))
	=0.
	\end{align*}
	Then assume that $\capa_1(\{x\})>0$.
	By Proposition \ref{prop:positive capacity implies thickness} we know
	that $\{x\}$ is $1$-thick at $x$, and so $x\notin A$.
	By Theorem
	\ref{thm:weak Cartan property} we find $R>0$ and
	functions $u_1,u_2\in\BV(X)$
	such that $\max\{u_1^{\wedge},u_2^{\wedge}\}=1$ in $A\cap B(x,R)$ and
	$u_1^{\vee}(x)=u_2^{\vee}(x)=0$.
	Then also $\max\{u_1^{\vee},u_2^{\vee}\}\ge 1$ in $A\cap B(x,R)$.
	Fix $0<\eps<\capa_1(\{x\})$.
	By Proposition \ref{prop:quasisemicontinuity of BV}
	there exists an open set $G\subset X$
	with $\capa_1(G)<\eps$
	such that $u_1^{\vee}|_{X\setminus G}$ is upper semicontinuous.
	By comparing capacities, we conclude that $x\notin G$.
	Thus by the upper semicontinuity, we necessarily have
	$\{u_1^{\vee}\ge 1\}\cap B(x,r)\subset G$
	for some $0<r<R_0/2$.
	This implies that $\capa_1(\{u_1^{\vee}\ge 1\}\cap B(x,r))<\eps$.
	Analogously, and by making $r$ smaller if necessary,
	$\capa_1(\{u_2^{\vee}\ge 1\}\cap B(x,r))<\eps$, so in total,
	\[
	\capa_1(A\cap B(x,r))\le \capa_1((\{u_1^{\vee}\ge 1\}\cup \{u_2^{\vee}\ge 1\})\cap B(x,r))
	<2\eps.
	\]
	Then by \eqref{eq:comparison of capacities},
	\[
	\rcapa_1(A\cap B(x,r),B(x,R_0))\le 2\left(1+\frac{2}{R_0}\right)\capa_1(A\cap B(x,r))
	< 4\eps\left(1+\frac{2}{R_0}\right).
	\]
	Since $\eps$ can be chosen arbitrarily small, we have the result.
\end{proof}

Just as in the case $p>1$, see \cite[Proposition 6.3]{BBL-WC}, at points of
nonzero capacity we obtain a strong Cartan property,
where we need only one superminimizer.

\begin{proposition}\label{prop:strong Cartan property}
Suppose that $x\in X$ with $\capa_1(\{x\})>0$, that $A\subset X$
is $1$-thin at $x$,
and that $0<R<\frac 18 \diam X$. Then
there exists a $1$-superminimizer $u$ in $B(x,R)$ such that
\[
\lim_{A\ni y\to x}u^{\wedge}(y)= \infty> u^{\vee}(x).
\]
\end{proposition}
\begin{proof}
By Proposition \ref{prop:cap goes to zero at thinness points}
we find a decreasing sequence of numbers $0<r_i<R$ such that
\[
\rcapa_1(A\cap B(x,r_i),B(x,R))<2^{-i},\quad i\in\N.
\]
Since $\rcapa_1$ is an outer capacity, there exist open sets $U_i\supset A\cap B(x,r_i)$
such that
\[
\rcapa_1(U_i,B(x,R))<2^{-i}.
\]
By the definition of the variational $1$-capacity, we find nonnegative
functions $\psi_i\in N^{1,1}(X)$ with $\psi_i= 1$ in $U_i$, $\psi_i=0$ in
$X\setminus B(x,R)$, and
\[
\int_X g_{\psi_i}\,d\mu<2^{-i},
\]
where as usual $g_{\psi_i}$ is the minimal $1$-weak upper gradient of $\psi_i$.
By the Sobolev inequality \eqref{eq:sobolev inequality}
and H\"older's inequality, we get
$\Vert \psi_i\Vert_{L^1(X)}<2^{-i}C_SR$, for each $i\in\N$.
By using the fact that $N^{1,1}(X)/\sim$ is a Banach space with the equivalence relation
$u\sim v$ if $\Vert u-v\Vert_{N^{1,1}(X)}=0$, see \cite[Theorem 1.71]{BB}, we conclude
\[
\psi:=\sum_{i=1}^{\infty} \psi_i\in N^{1,1}(X)\subset \BV(X)
\]
with $\psi=0$ in $X\setminus B(x,R)$.
Since $\psi \in \mathcal K_{\psi,0}(B(x,R))$,
by Proposition \ref{prop:existence of solutions} there exists a solution
$u$ of the $\mathcal K_{\psi,0}(B(x,R))$-obstacle problem.
Then $u$ is a $1$-superminimizer in $B(x,R)$ and
$u^{\wedge}\ge k$ in the open set
$U_1\cap \ldots \cap U_k\supset A\cap B(x,r_k)$, for every $k\in\N$.
Thus
\[
\lim_{A\ni y\to x}u^{\wedge}(y)=\infty.
\]
However, by \cite[Lemma 3.2]{KKST3} we know that $u^{\vee}(z)<\infty$ for $\mathcal H$-a.e. $z\in X$, and thus $u^{\vee}(z)<\infty$ for $1$-q.e. $z\in X$
by \eqref{eq:null sets of Hausdorff measure and capacity}.
Since $\capa_1(\{x\})>0$, necessarily $u^{\vee}(x)<\infty$.
\end{proof}

In the case $p>1$,
the $p$-fine topology is known to be the coarsest topology
that makes all $p$-superharmonic functions on open subsets of $X$ continuous, see
\cite[Theorem 1.1]{BBL-WC}.
Equivalently, it is the coarsest topology that makes such functions upper semicontinuous,
since they are lower semicontinuous already with respect to the metric topology.
In the following we consider what the analog of this could be in the case $p=1$.

\begin{definition}
We define the \emph{$1$-superminimizer topology} to be the coarsest topology
that makes the representative $u^{\vee}$ upper semicontinuous in $\Om$
for every $1$-superminimizer $u$ in $\Om$, for every open set
$\Om\subset X$.
\end{definition}

Note that if $X$ is bounded and thus compact, the only $1$-superminimizers in $X$ are constants
(for nonconstant $u\in\BV(X)$ we have $\Vert D\max\{u,k\}\Vert(X)<\Vert Du\Vert(X)$
for some $k\in\R$).
This is why we want to talk about $1$-super\-minimizers in open sets $\Om$, and as a result,
the metric topology is contained in the $1$-superminimizer topology by definition.

\begin{remark}
It would not make sense to replace $u^{\vee}$ by $u^{\wedge}$
in the definition of the $1$-superminimizer topology.
To see this, consider $X=\R$ (unweighted) and the Heaviside function
$u(x)=1$ for $x\ge 0$ and $u(x)=0$ for $x<0$. Moreover, let $v:=1-u$.
Now both $u$ and $v$ are clearly $1$-minimizers. On the other hand,
\[
\{u^{\wedge}<1\}\cap \{v^{\wedge}<1\}=\{0\}.
\]
Hence if the sets $\{u^{\wedge}<t\}$, for $t\in\R$ and $1$-superminimizers $u\in\BV(X)$,
are open in some topology, this topology contains \emph{all} subsets of $\R$.
\end{remark}

\begin{theorem}\label{thm:superminimizer top contains fine top}
The $1$-superminimizer topology contains the $1$-fine topology.
\end{theorem}
\begin{proof}
Let $U\subset X$ be a $1$-finely open set, and let $x\in U$.
The set $X\setminus U$ is $1$-thin at $x$. By Theorem \ref{thm:weak Cartan property},
there exist
$R>0$ and $1$-superminimizers $u_1,u_2$ in $B(x,R)$ such that
$\max\{u_1^{\vee},u_2^{\vee}\}\ge \max\{u_1^{\wedge},u_2^{\wedge}\}=1$ in $B(x,R)\setminus U$
and $u_1^{\vee}(x)=u_2^{\vee}(x)=0$.
Thus $x\in B(x,R)\cap \{u_1^{\vee}<1\}\cap \{u_2^{\vee}<1\}$,
which is a set belonging to the $1$-superminimizer topology,
and contained in $U$.
\end{proof}

Now it might seem
reasonable to postulate that the converse would hold as well, i.e. that the
$1$-fine topology would make $u^{\vee}$ upper semicontinuous
for all $1$-superminimizers $u$ in open sets.
However, this is not the case.

\begin{example}\label{ex:counterexample}
Let $X=\R^2$ with the usual $2$-dimensional Lebesgue measure $\mathcal L^2$,
let $0<\eps<1/5$, and let
\[
A:=\bigcup_{j=0}^{\infty} A_j
\]
with $A_j:=[10^{-j}-10^{-j}\eps,10^{-j}]\times [0,10^{-2j}\eps]$.
Denote the origin by  $0$.
It is straightforward to check that for any $0<R<1$,
\begin{equation}\label{eq:a thick set}
\frac{R\eps}{10}\le \rcapa_1(A\cap B(0,R),B(0,2R))\le 3R\eps,
\end{equation}
which is comparable to $\mathcal L^2(B(0,R))/R$.
Let $E\subset \R^2$ be a solution of the $\mathcal K_{A,0}(B(0,2))$-obstacle problem.
For any $y\in \R^2$ with $\tfrac{5}{16}\le |y|< \tfrac 25$,
by Lemma \ref{lem:smallness in annuli} and \eqref{eq:a thick set} we find
\[
\ch_E^{\vee}(y)\le C_4 \frac{\rcapa_1(A\cap B(0,1),B(0,2))}{\mathcal L^2(B(0,1))}
\le C_4\eps\le 1/2
\]
by choosing $\eps\le1/2C_4$.
Thus $\ch_E^{\vee}(y)=0$ for $\tfrac{5}{16}\le |y|<\tfrac 25$,
and so
\[
P(E,\R^2)=P(E\cap B(0,\tfrac{5}{16}),\R^2)+P(E\setminus B(0,\tfrac 25),\R^2).
\]
Thus we see that the minimization of the perimeter of $E$ (i.e. solving the obstacle problem)
takes place independently in the sets $B(0,\tfrac{5}{16})$ and
$E\setminus B(0,\tfrac 25)$.
Now it is straightforward to show that we must have $E\setminus B(0,\tfrac 25)=A_0$.
Inductively, we find $E=A$.
Clearly $\ch_A^{\vee}(0)=0$, but on the other hand, $A$ is $1$-thick at the origin, by
\eqref{eq:a thick set}. Thus $\ch_E^{\vee}$ is not $1$-finely upper semicontinuous at the origin.
\end{example}

Nevertheless, it is perhaps interesting to note that in Theorem
\ref{thm:weak Cartan property in text}, $\ch_{E_0}^{\vee}$ and $\ch_{E_1}^{\vee}$
\emph{are} $1$-finely upper semicontinuous at $x$, since
$\ch_{E_0}^{\vee}(x)=0=\ch_{E_1}^{\vee}(x)$
and the sets $\{\ch_{E_0}^{\vee}>0\}$ and $\{\ch_{E_1}^{\vee}>0\}$ are $1$-thin at $x$.
We expect this fact to be a useful substitute for fine upper semicontinuity
in future research.

In Table \ref{tab:comparison} we compare the properties of Newton-Sobolev and $p$-super\-harmonic
functions (for $1<p<\infty$) with the analogous properties of $\BV$
functions and $1$-superminimizers.
For the results in the left column, see the comment after
Proposition \ref{prop:quasisemicontinuity of BV},
the comment after Theorem \ref{thm:fine semicontinuity of BV functions},
\cite[Theorem 5.1]{KiMa} or \cite[Theorem 8.22]{BB}, and \cite[Theorem 1.1]{BBL-WC}.
For the results in the right column, see Proposition \ref{prop:quasisemicontinuity of BV},
Theorem \ref{thm:fine semicontinuity of BV functions},
Theorem \ref{thm:superminimizers are lsc}, and
Theorem \ref{thm:superminimizer top contains fine top}.

\paragraph{Acknowledgments.}
The research was
funded by a grant from the Finnish Cultural Foundation.
The author wishes to thank Nageswari Shan\-muga\-lingam
for helping to derive the lower semicontinuity property of $1$-super\-minimizers.

\begin{table}
\caption{A comparison chart.}
\begin{tabular}{ p{6cm} | p{6cm} }

	\emph{Properties of Newton-Sobolev and $p$-superharmonic functions, for $1<p<\infty$:} 
	& \emph{Properties of BV functions and 1-superminimizers:} \\ \hline
	$\bullet$ Every $u\in N^{1,p}(X)$ is quasicontinuous.
	&$\bullet$ For every $u\in\BV(X)$, $u^{\wedge}$ is quasi lower semicontinuous. \\
	$\bullet$ Every $u\in N^{1,p}(X)$ is $p$-finely continuous $p$-q.e.
	&$\bullet$ For every $u\in \BV(X)$, $u^{\wedge}$ is $1$-finely lower semicontinuous $1$-q.e. \\
	$\bullet$ Every $p$-superminimizer has a lower semicontinuous representative
	(a $p$-superharmonic function).
	&$\bullet$ For every $1$-superminimizer $u$, $u^{\wedge}$ is lower semicontinuous. \\
	$\bullet$ Any topology that makes $p$-superharmonic functions (upper semi-)con\-tinuous
	 in open sets contains the $p$-fine topology.
	&$\bullet$ Any topology that makes $u^{\vee}$ upper semicontinuous for
	every $1$-superminimizer $u$ in every open set
	contains the $1$-fine topology.\\
	$\bullet$ The $p$-fine topology makes $p$-superharmonic
	functions in open sets continuous.
	&  \multicolumn{1}{c}{ \parbox[t]{0cm}{\phantom\\ ?}}
\end{tabular}
\label{tab:comparison}
\end{table}

\noindent Address:\\

\noindent Department of Mathematics\\
Link\"oping University\\
SE-581 83 Link\"oping, Sweden\\
E-mail: {\tt panu.lahti@aalto.fi}

\end{document}